\documentclass{amsart}[12pt]
\usepackage[dvips]{graphicx}
\usepackage{lineno,cases}
\usepackage{algorithm}
\usepackage{algorithmic}
\usepackage{amsmath,amscd}

\newtheorem{prop}{\bf Proposition}[section]
\newtheorem{thm}[prop]{\bf Theorem}
\newtheorem{lem}[prop]{\bf Lemma}

\theoremstyle{definition}
\newtheorem{defi}[prop]{Definition}
\newtheorem{rmk}[prop]{Remark}

\newtheorem{condi}[prop]{Condition}

\title{Topological Entropy of Random Walks on Mapping Class Groups}
\author{Hidetoshi Masai}
\address{
Graduate School of Mathematical Sciences, 
The University of Tokyo,
3-8-1 Komaba Meguro-ku Tokyo 153-8914, Japan}
\email{masai at ms.u-tokyo.ac.jp}
\date{}

\begin{document}

\begin{abstract}
For any pseudo-Anosov diffeomorphism on a closed orientable surface $S$ of genus 
greater than one, it is known by the work of Bers and Thurston that the topological entropy agrees with the translation distance on the Teichm\"uller space with respect to the Teichm\"uller metric.
In this paper, we consider random walks on the mapping class group of $S$.
The drift of a random walk is defined as the translation distance of the random walk.
We define the topological entropy of a random walk and prove that it almost surely agrees with the drift on the Teichm\"uller space with respect to the Teichm\"uller metric.
\end{abstract}
\maketitle
\thispagestyle{empty}
\section{Introduction}
Let $S$ be a closed orientable surface of genus $\geq 2$. 
According to the Nielsen-Thurston classification \cite{Thu}, every non-periodic irreducible automorphism of $S$ is isotopic to a pseudo-Anosov diffeomorphism.
Thurston  proved that the topological entropy of any pseudo-Anosov diffeomorphism $\varphi$ coincides with $\log\lambda_{\varphi}$
where $\lambda_{\varphi}$ is the dilatation of $\varphi$ (c.f. \cite[Expos\'e 10]{FLP}).
Also by the work of Bers \cite{Ber}, $\log\lambda_{\varphi}$ is known to be equal to the translation distance of $\varphi$ on the Teichm\"uller space with respect to the Teichm\"uller metric.
The purpose of this paper is to demonstrate a ``random version'' of the work of Bers and Thurston.
Let $\mathrm{MCG}(S)$ denote the mapping class group of $S$.
We consider the random walk on $\mathrm{MCG}(S)$ which is determined by a probability measure $\mu$ on $\mathrm{MCG}(S)$.
This $\mu$ induces a probability measure $\mathbb{P}$ on $\mathrm{MCG}(S)^{\mathbb{Z}}$.
Throughout the paper we assume that $\mu$ has finite first moment with respect to the Teichm\"uller metric and the support of $\mu$ generates a non-elementary subgroup of $\mathrm{MCG}(S)$ (see Condition \ref{cond}).
Before stating the main theorem, we prepare several terminologies briefly.
Formal definitions are given in \S\ref{sec.pre}.
First, the topological entropy $h(\omega)$ of a sample path $\omega=(\omega_{n})\in\mathrm{MCG}(S)^{\mathbb{Z}}$ 
is defined using open coverings of $S$, similarly to the one for surface diffeomorphisms.
This measures growth rate of the number of distinguishable orbits of the random walk.
%
Next, Karlsson \cite{Kar} proved that for $\mathbb{P}$-a.e $\omega = (\omega_{n})$, the exponential growth rate of the length of the image $\omega_{n}(\alpha)$ of any simple closed curve $\alpha$ with respect to any metric always gives the same quantity, which is called the ``Lyapunov exponent'' $\lambda(\omega)$ of $\omega$.
Moreover, it is also proved that $\log\lambda(\omega)$ almost surely coincides with the drift $L_{a}(\omega)$ with respect to Thurston's asymmetric Lipschitz metric on the Teichm\"uller space $\mathcal{T}(S)$ of $S$.
Roughly speaking, the drift is the translation distance of $\omega$ .
The goal of this paper is to show that those quantities are the same almost surely.
\begin{thm}\label{thm.main}
Let $\mu$ be a probability measure on $\mathrm{MCG}(S)$ which satisfies Condition \ref{cond} and 
$\mathbb{P}$ the probability measure on $\mathrm{MCG}(S)^{\mathbb{Z}}$ induced by $\mu$.
For $\mathbb{P}$-a.e. $\omega\in\mathrm{MCG}(S)^{\mathbb{Z}}$,
we have the following equality.
$$L_{a}(\omega) = \log\lambda(\omega) = h(\omega) = L_{\mathcal{T}}(\omega),$$
where $L_{{\mathcal{T}}}(\omega)$ is the drift with respect to the Teichm\"uller metric.
These quantities are independent of $\omega$ and they are invariants of the random walk.
\end{thm}
The strategy of the proof is similar to the 
one for pseudo-Anosov diffeomorphisms in \cite[Expos\'e 9-10]{FLP}.
Indeed, $\log\lambda(\omega)\leq h(\omega)$ can be proved almost in the same way as the case of pseudo-Anosov diffeomorphisms.
To prove the opposite inequality for pseudo-Anosovs, in \cite{FLP}, a subshift of finite type is associated to the dynamics of a pseudo-Anosov iteration by constructing so called a Markov partition.
We will define a random subshift of finite type as a ``random version'' of a subshift of finite type (see \S\ref{sec.random subshift}).
Then 
 we will construct a semi-Markov partition of $S$ which respects the dynamics of $\omega$ (see Definition \ref{defi.semi-Markov}) and associate to it a random subshift of finite type.
The main difficulty, unlike pseudo-Anosov diffeomorphisms, is that for $\mathbb{P}$-a.e $\omega=(\omega_n)\in\mathrm{MCG}(S)^{\mathbb{Z}}$,
some part of $\omega$ can be arbitrarily ``bad''.
For example, the orbit $\{\omega_{n}X\}$ of any point $X\in\mathcal{T}(S)$ may have large backtrack.
On the other hand, for a pseudo-Anosov $\varphi$ case, the fact that $\varphi$ acts as a translation on a Teichm\"uller geodesic is implicitly used in \cite{FLP}.
To overcome the difficulty, in \S \ref{sec.Markov}-\ref{sec.main}, we show that it suffices to observe only 
``good'' elements in an orbit. 
The existence of such ``good'' elements follows from ergodic theorems.

To consider dynamics of $\omega$ on the surface, we need to take representatives of mapping classes.
Let $\mathrm{Diff}^{+}(S)$ denote the space of orientation preserving diffeomorphisms on $S$.
Let $w_{n}\in\mathrm{Diff}^{+}(S)$ be a representative of $\omega_{n}$ and $\bold{w} := (w_{n})_{n\in\mathbb{Z}}$.
Another difficulty occurs after taking representatives, that is, we can not use ergodic theorems.
This is because we can not take representatives so that they are compatible with the shift maps on $\mathrm{MCG}(S)^{\mathbb{Z}}$, denoted $\theta$.
Notations $(\bold{w},n)$ are used instead of $\theta^{n}\omega$ to warn readers this issue.
Our goal in \S \ref{sec.Markov} is to prove the following theorem.
\begin{thm}\label{thm.symbolic}
There exist a random subshift of finite type $(\{\Sigma_{A}(\bold{w},n)\}_{n\in\mathbb{Z}},\sigma)$ such that the following diagram commutes for any 
$n\in\mathbb{Z}$.
$$
\begin{CD}
\Sigma_{A}(\bold{w},0) @>\sigma^{n} >> \Sigma_{A}(\bold{w},n)\\
@Vp(\bold{w},0)VV @VVp(\bold{w},n)V \\
S@>\omega_{n}^{-1} >>S
\end{CD}
$$
where $p(\bold{w},n):\Sigma(\bold{w},n)\rightarrow S$ is a continuous surjective map and $\sigma$ is the shift map.
\end{thm}

The topological entropy of $\sigma$ in Theorem \ref{thm.symbolic} can be defined as the growth rate of the number of cylinder sets of length $n$ in $\Sigma_A(\bold{w},0)$.
Let $h(\sigma)$ denote the topological entropy of $\sigma$.
The fact $h(\sigma)\leq L_{\mathcal{T}}$ can be easily observed (Lemma \ref{lem.p-entropy}).
For pseudo-Anosovs, the facts of type Theorem \ref{thm.symbolic} and $h(\sigma)\leq L_{\mathcal{T}}$ suffice to prove a theorem of type Theorem \ref{thm.main}.
However, we need to vary structures of $S$ to construct  a semi-Markov partition.
Hence we need to discuss how structures, especially the Lebesgue number of a fixed open covering, vary as the steps.
The Lebesgue numbers are discussed in \S \ref{sec.Lebesgue}, 
and the rest of \S \ref{sec.main} is devoted for a proof of $h(\omega) \leq h(\sigma)$.

\section{Preliminaries}\label{sec.pre}
In this section, we prepare terminologies and basic facts which we need to prove Theorem \ref{thm.main}.

\subsection{ Teichm\"uller space}
We briefly recall the Teichm\"uller spaces and related facts.
Readers should refer \cite{FM,FLP} for more details.
Let $S$ be a closed orientable surface of genus $g(S)>1$.
A {\em marked Riemann surface} is a pair of a Riemann surface 
$\mathcal{X}$ and a homeomorphism, called a {\em marking}, $f:S\rightarrow \mathcal{X}$.
Two marked Riemann surfaces $(\mathcal{X}_{i}, f_{i}:S\rightarrow \mathcal{X}_{i}),~(i = 1,2)$ are said to be {\it Teichm\"uller equivalent}
if there is a biholomorphic map $\phi:\mathcal{X}_{1}\rightarrow\mathcal{X}_{2}$ such that $\phi\circ f_{1}$ is homotopic to $f_{2}$.
The Teichm\"uller space $\mathcal{T}(S)$ of $S$ is the space of marked Riemann surfaces modulo Teichm\"uller equivalence.
Since each marking defines a complex structure on $S$ by pullback,
we often confuse a point $X\in\mathcal{T}(S)$ with a complex structure on $S$.
The mapping class group $\mathrm{MCG}(S)$ acts on $\mathcal{T}(S)$ so that for $\varphi\in\mathrm{MCG}(S)$
with a representative $\psi\in\mathrm{Diff}^{+}(S)$,
$\varphi\cdot(\mathcal{X}, f:S\rightarrow\mathcal{X}) = 
(\mathcal{X}, f\circ\psi^{-1}:S\rightarrow\mathcal{X})$.

A holomorphic quadratic differential on $X\in\mathcal{T}(S)$ is a family of holomorphic maps $q = \{q_{\alpha}\}$ each defined 
on $z_{\alpha}(U_{\alpha})$ of a complex chart $U_{\alpha}\subset X,~z_{\alpha}:U_{\alpha}\rightarrow \mathbb{C}$ so that if 
$U_{\alpha}\cap U_{\beta}\not=\emptyset$ then 
$$q_{\beta}(z_{\beta}) = q_{\alpha}\circ z_{\alpha\beta}(z_{\beta})\cdot(z_{\alpha\beta}'(z_{\beta}))^{2}$$ where
$z_{\alpha\beta} := z_{\alpha}\circ z_{\beta}^{-1}$.
For $X\in\mathcal{T}(S)$, let $Q(X)$ denote the space of quadratic differentials.
The vertical (resp. horizontal) trajectories of a quadratic differential $q$ are curves $z(t)$ such that 
$q(z(t))z'(t)^{2}\in \mathbb{R}_{>0}$ (resp. $\mathbb{R}_{<0}$).
For each smooth arc $\tau$, 
the transverse measures on the vertical and horizontal trajectories are defined by
$\int_{\tau}|\Im q(z)^{1/2}dz|$ and $\int_{\tau}|\Re q(z)^{1/2}dz|$ respectively.
Thus each $q\in Q(X)$ defines two measured foliations called {\em vertical } and {\em horizontal} foliations as
the vertical and horizontal trajectories equipped with the transverse measures respectively.
A theorem of Teichm\"uller says that given two points $X,Y\in\mathcal{T}(S)$, there exists a quasi-conformal map $T:X\rightarrow Y$ and quadratic differentials $q_{X}\in Q(X)$ and $q_{Y}\in Q(Y)$ 
such that the map $T$ maps $q_{X}$ to $q_{Y}$ so that it stretches (resp. contracts) the horizontal (resp. vertical) foliations. 
The logarithm of the stretch factor coincides with the Teichm\"uller distance $d_{\mathcal{T}}(X,Y)$.
By integrating the square root of a quadratic differential $q$, we have a singular Euclidean metric on $X$.
With respect to the singular Euclidean metric,
 the length of a smooth arc $\tau'$, denoted by $|\tau'|_{q}$, is equal to 
$\int_{\tau'}|q^{1/2}dz|$.
For $q\in Q(X)$, we define the norm of $q$ by $||q|| := \int\int_{X}|q|$.
Let $Q_{1}(X):=\{q\in Q(X)\mid ||q|| = 1\}$.
We denote by $\mathcal{PMF}(S)$ the space of projective measured foliations.
We consider the Thurston compactification $\bar{\mathcal{T}}(S):=\mathcal{T}(S)\cup\mathcal{PMF}(S)$ on which
$\mathrm{MCG}(S)$ acts continuously.
By the work of Thurston, $\mathcal{PMF}(S)$ is homeomorphic to the sphere of 
dimension $6g(S)-6$ \cite{Thu}.
We now recall the work of Hubbard-Masur.
\begin{thm}[\cite{HM}]\label{thm.HM}
The map $Q_{1}(X)\rightarrow\mathcal{PMF}(S)$ associating the equivalence class of 
the horizontal foliation to each $q\in Q_{1}(X)$ is a homeomorphism.
\end{thm}
Let $F,G\in\mathcal{PMF}(S)$ be transverse filling projective measured foliations.
Let $\Gamma(F,G)\subset \mathcal{T}(S)$ denote the Teichm\"uller geodesic corresponding to a quadratic differential with horizontal and vertical foliation $F$ and $G$ respectively (see \cite{GM} for the existence of such geodesics).
A projective measured foliation is called {\em uniquely ergodic} if its supporting foliation admits only one transverse measure up to scale.
Let $\mathcal{UE}(S)\subset \mathcal{PMF}(S)$ denote the space of uniquely ergodic foliations.

\subsection{Random walk on group}
Let $G$ be a countable group and $\mu:G\rightarrow [0,1]$ a probability measure.
By $\mathbb{Z}_{+}$ (resp. $\mathbb{Z}_{-}$), we denote the space of positive (resp. negative) integers.
For group elements $x_{1}, \dots, x_{n}\in G$, the subset $$[x_{1}, \dots, x_{n}]:=\{\omega = (\omega_{i})\in G^{\mathbb{Z}_{+}}\mid \omega_{i} = x_{i} \text{ for } 1\leq i\leq n\}$$
is called a {\em cylinder set}.
The probability measure $\mu$ induces a probability measure $\mathbb{P}$ on the space of sample paths $G^{\mathbb{Z}_{+}}$ so that 
$$\mathbb{P}([x_{1}, \dots, x_{n}]) = \mu(x_{0}^{-1}x_{1})\mu(x_{1}^{-1}x_{2})\cdots\mu(x_{n-1}^{-1}x_{n}),$$
where $x_{0}$ is the initial element which is assumed to be the identity unless otherwise stated.
We also consider the {\em reflected measure} $\check\mu(g):=\mu(g^{-1})$.
Let $\check{\mathbb{P}}$ be the probability measure on $G^{\mathbb{Z}_{-}}$ induced by $\check\mu$.
Then by the map $\omega = (\omega_{n})_{n\in\mathbb{Z}}\mapsto ((\omega_{n})_{n\in\mathbb{Z}_{+}}, (\omega_{n})_{n\in\mathbb{Z}_{-}})$, 
the probability measure $\mathbb{P}\times\check{\mathbb{P}}$ induces a probability measure on $G^{\mathbb{Z}}$
 which we again denote by $\mathbb{P}$.
We define the {\em Bernoulli shift}, denoted by $\theta$, as for any $k\in\mathbb{Z}$,
$$(\theta^{k} \omega)_{n} := \omega_{k}^{-1}\omega_{n+k} \text{,  }\forall n\in\mathbb{Z}.$$

Recall that a subgroup of $\mathrm{MCG}(S)$ is called {\em non-elementary} 
if it contains two pseudo-Anosov elements with disjoint fixed point sets in $\mathcal{PMF}(S)$.
From now on, we consider the random walk on $\mathrm{MCG}(S)$ which is determined by a probability measure $\mu$ which satisfies the following condition.
\begin{condi}\label{cond}
The probability measure $\mu:\mathrm{MCG}(S)\rightarrow [0,1]$ satisfies that
\begin{itemize}
\item $\mu$ has finite first moment with respect to the Teichm\"uller metric on $\mathcal{T}(S)$ 
i.e. for any $X\in\mathcal{T}(S)$, $\sum_{g\in\mathrm{MCG}(S)} \mu(g)d_{\mathcal{T}}(X, gX)<\infty$, and
\item the support of $\mu$ generates a non-elementary subgroup of $\mathrm{MCG}(S)$.
\end{itemize}
\end{condi}

\subsection{Topological entropy, drift, and Lyapunov exponent}
Let $\mathcal{A}=\{\mathcal{A}_{i}\}_{i\in I}$ and
$\mathcal{B}=\{\mathcal{B}_{j}\}_{j\in J}$ be open coverings of $S$.
Since $S$ is compact, each open covering has a finite subcover.
Let $N(\mathcal{A})$ denote the number of sets in a subcover of $\mathcal{A}$ with minimal cardinality.
$\mathcal{B}$ is said to be a {\em refinement} of a cover $\mathcal{A}$, denoted $\mathcal{A}\prec\mathcal{B}$,
 if for any $B\in\mathcal{B}$, there is $A\in\mathcal{A}$ such that $B\subset A$.
 It can readily be seen that if $\mathcal{A}\prec\mathcal{B}$ then $N(A)\leq N(B)$.
We denote by $\mathcal{A}\vee\mathcal{B}$ the open cover $\{A_i\cap B_j\}_{i\in I, j\in J}$.
For an open covering $\mathcal{A}$ of $S$ and a metric $d$ on S, the {\em Lebesgue number} $\delta_{d}(\mathcal{A})$ with respect to $d$
is defined to be
$$\inf_{x\in S}\sup\{r>0\mid B_{r}(x)\subset A \text{ for some } A\in \mathcal{A}\},$$
where $B_{r}(x)$ is the open ball centered at $x$ of radius $r$ with respect to $d$.
Let $r>0$ be less than the Lebesgue number of $\mathcal{A}$.
Then the covering consisting of all open balls of radius $r$ refines $\mathcal{A}$.

\begin{defi}[Topological entropy. c.f.\cite{AKM}]
Let $\omega = (\omega_n)_{n\in\mathbb{Z}}\in \mathrm{MCG}(S)^{\mathbb{Z}}$.
We first choose an arbitrary representative $w_{n}\in \mathrm{Diff}^{+}(S)$ of $\omega_{n}$ for each $n\in\mathbb{Z}$.
Let $\bold{w}:=(w_{n})_{n\in\mathbb{Z}}$.
For an open cover $\mathcal{A}$, let $N_n(\bold{w},\mathcal{A}):= N(\mathcal{A}\vee w_1(\mathcal{A})\vee\cdots\vee w_{n-1}(\mathcal{A}))$.
We define
$$h(\bold{w},\mathcal{A}) := \limsup_{n\rightarrow\infty}\frac{1}{n}\log N_n(\bold{w},\mathcal{A}).$$
Note that if $\mathcal{A}\prec\mathcal{B}$, then $h(\bold{w},\mathcal{A})\leq h(\bold{w},\mathcal{B})$.
The {\em topological entropy} of $\bold{w}$ is 
$$h(\bold{w}) := \sup_{\mathcal{A}} h(\bold{w},\mathcal{A}),$$
where the supremum is taken over all open coverings of $S$.
Finally we define $$h(\omega):= \inf_{\bold{w}}h(\bold{w})$$ where the infimum is taken over all representatives of $\omega$.
\end{defi}
\begin{rmk}
Unlike the definition of topological entropy of surface automorphisms, we do not take inverses.
This is natural because when we consider random walks, we multiply new elements from the right.
\end{rmk}
%

We define the drift of random walks, which we may regard as a ``translation distance'' of the random walk.
\begin{defi}
Let $(X,d_{X})$ be a metric space on which $\mathrm{MCG}(S)$ acts isometrically.
Suppose the probability measure $\mu$ has finite first moment with respect to $d_{X}$, i.e.
$$\sum_{g\in\mathrm{MCG}(S)}\mu(g)d_{X}(x,gx)<\infty,$$
where $x\in X$ is arbitrary.
By Kingman's subadditive ergodic theorem, the limit
$$\lim_{n\rightarrow \infty} \frac{1}{n}d_{X}(x,\omega_{n}x)$$
exists for $\mathbb{P}$-a.e. $\omega$ and this limit is independent of $x$ and $\omega$.
This limit is called the {\em drift} of $\omega\in\mathrm{MCG}(S)^{\mathbb{Z}}$ with respect to $d_{X}$.
\end{defi}

Let $d_{a}$ and $d_{\mathcal{T}}$ denote the distance on $\mathcal{T}(S)$ by Thurston's Lipschitz metric and the Teichm\"uller metric respectively.
We here recall the work of Choi-Rafi. 
\begin{thm}[{\cite[Theorem B]{CR}}]\label{thm.CR}
There is a constant $c$ depending on the surface $S$ and on $\delta$ such that for any $X, Y$
in the $\delta$-thick part of $\mathcal{T}(S)$,
$d_{\mathcal{T}} (X, Y)$ and $d_{a}(X, Y)$
differ from one another by at most $c$.
\end{thm}
By Theorem \ref{thm.CR},
if the probability measure $\mu$ has finite first moment with respect to the Teichm\"uller metric,
then it also has finite first moment with respect to Thurston's Lipschitz metric.
Let $L_{a}$ (resp. $L_{\mathcal{T}}$) denote the drift of $\omega$ with respect to $d_{a}$ (resp. $d_{\mathcal{T}}$).
Since the drifts are also independent of the choice of base points, by taking a point in the thick part of $\mathcal{T}(S)$, 
we have $L_{a} = L_{\mathcal{T}}$ by Theorem \ref{thm.CR}.
We let $L:= L_{a} = L_{\mathcal{T}}$.

In \cite{Kar}, Karlsson proved the following.
\begin{thm}[\cite{Kar}]\label{thm.Kar}
There exists $\lambda$ such that for $\mathbb{P}$-a.e. $\omega\in \mathrm{MCG}(S)^\mathbb{Z}$,
 for any isotopy class $\alpha$ of essential simple closed curves and Riemannian metric $\rho$ of $S$,
 $$\lim_{n\rightarrow\infty}\frac{1}{n}\log l_{\rho}(\omega_n^{-1}\alpha) = \log\lambda,$$
where $l_\rho(\alpha)$ denotes the infimum of the length of curves in $\alpha$ with respect to $\rho$.
Moreover $\log\lambda$ coincides with $L$.
\end{thm}
Note that for Theorem \ref{thm.Kar}, we do not need to take a representative of $\omega$.
Following \cite{DH}, we call $\lambda$ in Theorem \ref{thm.Kar} the {\em Lyapunov exponent} of the random walk.

Now we establish the following inequality.
\begin{lem}\label{lem.lemma1}
Let $\lambda$ be the Lyapunov exponent of the random walk determined by $\mu$.
For $\mathbb{P}$-a.e. $\omega$, we have
$$\log\lambda\leq h(\omega).$$
\end{lem}
\begin{proof}
We first fix a hyperbolic metric $\rho$ on $S$, a universal covering $\pi:\mathbb{H}^{2}\rightarrow S$ and a representative $\bold{w} = (w_{n})$ of $\omega$.
We also fix $p\in S$ and $\tilde{p}\in \pi^{-1}(p)$ in order to choose lifts $\widetilde w_{n}$ of $w_{n}$ uniquely for all $n\in\mathbb{Z}$.
In \cite{FLP}, a pseudo-Anosov diffeomorphism $\varphi:S\rightarrow S$ is discussed.
One can prove the following lemma by exchanging $\varphi^{n}$ with $w_{n}^{-1}$, and following the same argument as in \cite{FLP}.
\begin{lem}[c.f.{\cite[Lemma 10.8]{FLP}}]\label{lem.10.8}
For any $x,y\in\mathbb{H}^{2}$ 

$$\lim_{n\rightarrow\infty}\frac{1}{n}\log d_{\rho}(\widetilde w_{n}^{-1}x,\widetilde w_{n}^{-1}y)\leq h(\bold{w}).$$
\end{lem}
We may choose $x$ and $y$ in Lemma \ref{lem.10.8} to be the endpoints of a lift of geodesic representative of a simple closed curve $\alpha$ on $S$.
Then since $l_{\rho}(w_{n}^{-1}\alpha)\leq d_{\rho}(\widetilde w_{n}^{-1}x,\widetilde w_{n}^{-1}y)$,  we have $\log\lambda\leq h(\bold{w})$ for any representative $\bold{w}$ of $\omega$.
\end{proof}
In order to prove $h(\omega)\leq L(=\log\lambda)$, we need a notion of random subshift of finite type.

\subsection{Random subshift of finite type}\label{sec.random subshift}
We define a random subshift of finite type which we use to prove Theorem \ref{thm.main}.
Our goal is to associate a random subshift of finite type to a sample path $\omega\in\mathrm{MCG}(S)^{\mathbb{Z}}$.
Since we have to overcome certain difficulty which is described briefly in the introduction,
we need to slightly modify the definition from the standard one (see e.g. \cite[Definition 3.9]{GK} for the standard one).
The main difference is that we can only associate a random subshift of finite type to a representative $\bold{w}$ of $\omega\in\mathrm{MCG}^{\mathbb{Z}}(S)$.
For later convenience, we use the notations with $\bold{w}$ here.
\begin{defi}
Let $k(\bold{w},\cdot):\mathbb{Z}\rightarrow\mathbb{Z}$ be a function.
Suppose we have a family of $k(\bold{w},n)\times k(\bold{w},n+1)$ matrices $A(\bold{w},n)$ each of whose entry is $0$ or $1$. 
For any $n\in\mathbb{Z}$, let 
$$\Sigma_{k}(\bold{w},n) := \prod_{i\in\mathbb{Z}}\{1,2,\cdots, k(\bold{w},i+n)\}.$$
We define the coordinate so that for each element $(x_{i})\in\Sigma_{k}(\bold{w},n)$,
we have $x_{i}\in \{1,\cdots, k(\bold{w},i+n)\}$.
A random subshift of finite type is a pair $(\{\Sigma_{A}(\bold{w},n)\}_{n\in\mathbb{Z}},\sigma)$ where
$$ \Sigma_{A}(\bold{w},n) := \{ x = (x_{i})\in \Sigma_{k}(\bold{w}) \mid A(\bold{w},i+n)_{x_{i}, x_{i+1}} = 1 \text{ for all } 
i\in\mathbb{Z}\}$$
and $\sigma:\Sigma_{A}(\bold{w},n)\rightarrow\Sigma_{A}(\bold{w},n+1)$ is the standard left shift.
\end{defi}
We consider the discrete topology on each $\{1,\dots,k(\bold{w},n)\}$ and the product topology on $\Sigma(\bold{w},n)$.
Let $y_{i}\in k(\bold{w},i+n)$ for $s\leq i \leq t$.
The {\em $(s,t)_{n}$-cylinder set} in $\Sigma_{A}(\bold{w},n)$ of $(y_{i})$ is
$$\{x = (x_{i})\in\Sigma_{A}(\bold{w},n)\mid x_{i} = y_{i}, \text{ for all }s\leq i \leq t\}.$$
Let $\mathcal{C}_{n}(s,t)$ denote the family of $(s,t)_{n}$-cylinder sets in $\Sigma_{A}(\bold{w},n)$.
Note that $\sigma(\mathcal{C}_{n}(s,t)) = \mathcal{C}_{n+1}(s-1,t-1)$.

%
%
%
%
%
%
%

\section{Construction of semi-Markov partitions}\label{sec.Markov}

\subsection{Birectangle partition}

A semi-Markov partition with respect to $\omega\in\mathrm{MCG}(S)^{\mathbb{Z}}$ is 
a sequence of partitions of the surface $S$ by {\em birectangles} with certain condition so that it respects the dynamics of $\omega$ (see \S\ref{subsec.Markov}).
We first construct a birectangle decomposition, denoted $\mathcal{R}(F_{+}, F_{-}, \mathcal{X})$, from two transverse uniquely ergodic foliations $F_\pm\in \mathcal{UE}(S)$, and a marked Riemann surface $f:S\rightarrow\mathcal{X}$ which represents a point $X$ on the Teichm\"uller geodesic $\Gamma(F_{+}, F_{-})$.
The Riemann surface structure $\mathcal{X}$ lets us fix measured foliation representatives $(\mathcal{F}_{+},\mu_{+})$ and $(\mathcal{F}_{-},\mu_{-})$ of $F_{+}$ and $F_{-}$ respectively so that 
$(\mathcal{F}_{+},\mu_{+})$ and $(\mathcal{F}_{-},\mu_{-})$ are the horizontal and vertical foliation of a holomorphic quadratic differential on $\mathcal{X}$ of norm 1.
Their preimages by $f$ on $S$ are also denoted by the same notations.
Let $\mathrm{Sing}(\mathcal{F})$ denote the set of singular points of $\mathcal{F}$.
Note that with these representations, $\mathrm{Sing}(\mathcal{F}_{+}) = \mathrm{Sing}(\mathcal{F}_{-})$.
\begin{defi}
A subset $R\subset S$ is called an {\em $(\mathcal{F}_{+},\mathcal{F}_{-})$-rectangle}, or a {\em birectangle} if
$R$ is the image of some continuous map $\varphi:[0,1]\times[0,1]\rightarrow S$ such that 
\begin{itemize}
\item $\varphi|_{(0,1)\times (0,1)}$ is an embedding, and
\item for all $t\in[0,1]$, $\varphi([0,1]\times\{t\})$(resp.  $\varphi(\{t\}\times[0,1])$) is a finite union of leaves and singularities of $\mathcal{F}_{+}$
 (resp. $\mathcal{F}_{-}$), and in fact in one leaf if $t\in(0,1)$.
\end{itemize}
We let $\mathrm{int}(R) := \varphi((0,1)\times(0,1))$, 
$\partial_{h}R:=\varphi([0,1]\times \{0,1\})$,
$\partial_{v}R:=\varphi(\{0,1\}\times [0,1])$, and
$\partial R:=\partial_{h}R\cup \partial_{v}R$.

A family of birectangles $\mathcal{R} = \{R_{i}\}$ is called a {\em birectangle partition} if
\begin{enumerate}
\item $\bigcup_{i} R_{i} = S$, and
\item $\mathrm{int}(R_{i})\cap\mathrm{int}(R_{j}) = \emptyset$ for $i\not=j$.
\end{enumerate}
\end{defi}

For a singular measured foliation, we call a leaf which departs from a singularity a {\em singular leaf}.
Any small neighborhood of a singular point is decomposed into several components by singular leaves.
We call each component a {\em sector}.
A {\em saddle connection} is a singular leaf which connects two singular points.

We will now construct a birectangle partition of $S$.
We imitate the construction in \cite[Expos\'e  9]{FLP}.
For each sector of $\mathcal{F}_{-}$ of a singular point, we take a subarc of the singular leaf of $\mathcal{F}_{+}$ in the sector, which starts from the singular point
and have $\mu_{-}$ measure 1.
If $\mathcal{F}_{+}$ has a saddle connection and we can not take a singular leaf of $\mu_{-}$ measure 1, we instead take the whole saddle connection.
Let $\tau' = \tau'(F_{+},F_{-},\mathcal{X})\subset \mathcal{F}_{+}$ denote the family of such subarcs and saddle connections.
Then for each singular leaf of $\mathcal{F}_{-}$, 
we take the shortest subarc that starts from a singular point and intersects every element of $\tau'$ which is not a saddle connection at least once.
Similarly to before, we take whole saddle connections if there are no such subarcs.
Let $\eta'=\eta'(F_{+},F_{-},\mathcal{X})$ denote the family of such subarcs and saddle connections.
Then, for each $\alpha'\in\tau'$, we truncate the component of $\alpha'\setminus\eta'$ which contains $\partial\tau'\setminus \mathrm{Sing}(\mathcal{F}_{+})$ from $\alpha'$, and denote by $\alpha$ the resulting arc.
Note that saddle connections remain unchanged.
Let $\tau = \tau(F_{+},F_{-},\mathcal{X}) := \{\alpha\mid \alpha'\in\tau'\}$.
Then we extend each element of $\eta'$ until it meets $\tau$ exactly once more.
Let $\eta = \eta(F_{+},F_{-},\mathcal{X})$ denote the family of resulting subarcs.
Then we let $$\mathcal{R}(F_{+}, F_{-}, \mathcal{X}):= \{\overline{C}\mid C \text{ is a component of } S\setminus(\tau\cup\eta)\}.$$
\begin{lem}
$\mathcal{R}(F_{+}, F_{-}, \mathcal{X})$ is a birectangle partition.
\end{lem}
\begin{proof}
It suffices to prove that each element of $R\in\mathcal{R}(F_{+}, F_{-}, \mathcal{X})$ is a birectangle.
If $\partial R$ contains $\tau$, then by construction $R$ does not contain singular points in the interior.
By the singular Euclidean structure determined by $\mathcal{F}_{\pm}$, we see that two components of $\partial R\cap \eta$ are parallel and in particular 
$R$ is a birectangle.
If there were $R$ with $\partial R \cap\tau = \emptyset$, then $\partial R$ must have contained a loop consisting of leaves of $\mathcal{F}_{-}$.
However, since $\mathcal{F}_{-}$ is uniquely ergodic, there are no such loops.
\end{proof}
%
%

%
\subsection{Semi-Markov partition}\label{subsec.Markov}
Let $\omega\in\mathrm{MCG}(S)^{\mathbb{Z}}$ and $\bold{w} = (w_{n})$ be a representative of $\omega$. 
Our goal in this subsection is to construct a semi-Markov partition from a birectangle partition obtained in the previous subsection, so that
it respects the dynamics of $\bold{w}$.
\begin{defi}\label{defi.semi-Markov}
A sequence of birectangle partitions $\{\mathcal{R}_{n}\}_{n\in\mathbb{Z}}$ is a {\em semi-Markov partition} with respect to $\bold{w}$ if for every $n\in\mathbb{Z}$,
\begin{enumerate}
\item[(M1)] $w_{n}\partial_h \mathcal{R}_{n} \subset w_{n+1}\partial_{h} \mathcal{R}_{n+1}$,  $w_{n}\partial_{v} \mathcal{R}_{n} \supset w_{n+1}\partial_{v} \mathcal{R}_{n+1}$, and 
\item[(M2)] for each $R_{n}\in\mathcal{R}_{n}$ and $R_{n+ 1}\in\mathcal{R}_{n+ 1}$, if $w_{n} R_{n}$ and $w_{n+1}R_{n+ 1}$ intersects, then 
the intersection is a single birectangle.
\end{enumerate}
\end{defi}
We call it a {\em semi}-Markov partition because to have $h(\omega)\leq L$, we further need estimates for the size of birectangles.
We here carefully construct a semi-Markov partition so that it further satisfies certain estimates which we give in \S \ref{sec.main}.

Recall that a Markov partition for a pseudo-Anosov diffeomorphism $\varphi$ is constructed by using stable and unstable foliations $F_{s}$ and $F_{u}$ of $\varphi$.
In terms of the Thurston compactification $\bar{\mathcal{T}}(S)$, 
these foliations are characterized as limits 
$$\lim_{n\rightarrow\infty} \varphi^{n}X = F_{s} \text{ and }
\lim_{n\rightarrow\infty} \varphi^{-n}X = F_{u}$$
 where $X\in\mathcal{T}(S)$ is an arbitrary point.
Kaimanovich-Masur proved that for the case of random walks, we have similar limits.
\begin{thm}[{\cite[Theorem 2.2.4]{KM}}]\label{thm.KM}
Let $\mu$ be a probability measure which satisfies Condition \ref{cond}.
 Then
\begin{enumerate}
 \item  There exists a unique $\mu$-stationary probability measure $\nu$
	on $\mathcal{PMF}(S)$ 
	which is purely non-atomic and concentrated on $\mathcal{UE}(S)$.
\item  For $\mathbb{P}$-a.e. $\omega\in\mathrm{MCG}(S)^{\mathbb{Z}_{+}}$ and any $X\in\mathcal{T}(S)$,
       the sequence $\omega_nX$ converges in $\mathcal{PMF}(S)$ to a limit 
       $F(\omega)\in\mathcal{UE}(S)$ 
       and the distribution of the limits is given by the measure $\nu$.
\end{enumerate}
\end{thm}
We may apply Theorem \ref{thm.KM} both to $\mu$ and $\check\mu$.
We denote by $\check\nu$ the $\check\mu$-stationary measure on $\mathcal{PMF}(S)$.
For $\mathbb{P}$-a.e $\omega\in\mathrm{MCG}(S)^{\mathbb{Z}}$, let 
$$F_{+}(\omega):= \lim_{n\rightarrow+\infty}\omega_{n}X \text{, and } F_{-}(\omega):= \lim_{n\rightarrow-\infty}\omega_{n}X.$$
Let $\Gamma(\omega)$ denote the Teichm\"uller geodesic $\Gamma(F_{+}(\omega), F_{-}(\omega))$.
Note that by the definition of the Bernoulli shift $\theta$, we have $F_{+}(\theta^{n}\omega) = \omega_{n}^{-1}F_{+}(\omega)$ and
$F_{-}(\theta^{n}\omega) = \omega_{n}^{-1}F_{-}(\omega)$, and hence
$\Gamma(\theta^{n}\omega) = \omega_{n}^{-1}\Gamma(\omega)$.
We first fix $X_{0}$ on $\Gamma(\omega)$.
By \cite[Lemma 1.4.3]{KM}, the function $D:\mathcal{PMF}(S)\times \mathcal{PMF}(S)\rightarrow \mathbb{R}$, 
$(G_{+},G_{-})\mapsto d_{\mathcal{T}}(X_{0}, \Gamma(G_{+},G_{-}))$ is continuous where it is defined.
We fix open neighborhoods $U_{+}$ of $F_{+}(\omega)$ and $U_{-}$ of $F_{-}(\omega)$ with the following condition.
\begin{condi}\label{condi.U}
The neighborhoods $U_{\pm}$ satisfy
\begin{itemize}
\item $U_{+}$ and $U_{-}$ have positive $\nu$ and $\check\nu$ measure respectively, and
\item  for any $G_{+}\in U_{+}$ and $G_{-}\in U_{-}$, there is the Teichm\"uller geodesic $\Gamma[G_{+},G_{-}]$.
\item  $D(U_{+}, U_{-})$ is bounded from above by some constant $C>0$.
\end{itemize}
\end{condi}
The construction of semi-Markov partition in this section works for any $U_{\pm}$ satisfying Condition \ref{condi.U}.
We give $U_{\pm}$ which satisfy further condition that we need to prove Theorem \ref{thm.main} in \S \ref{sec.refine}.
Finally let $\delta>0$ be small enough so that the $C$-neighborhood of $X_{0}$ is contained in the $\delta$-thick part of $\mathcal{T}(S)$.

We now choose points in $\mathcal{T}(S)$ to construct a semi-Markov partition.
See Figure \ref{fig.position} for a schematic picture.
First, we choose $X'_{n}$ to be a closest point to $X_{0}$ on $\Gamma(\theta^{n}\omega)$.
For $n$ positive, we define $X''_{n}\in\Gamma(\theta^{n}\omega)$ inductively by 
\begin{subnumcases}
{X''_{n} := }\nonumber
X'_{n} & if $F_{+}(\theta^{n}\omega)\in U_{+}$ and $F_{-}(\theta^{n}\omega)\in U_{-}$\\
\omega_{n}^{-1}\omega_{n-1}X''_{n-1}, & otherwise.
\end{subnumcases}
We define $X''_{n}$ for negative $n$ similarly.
We then define $\varepsilon:\mathbb{Z}\rightarrow \{0,1\}$ as follows.
For positive $n$,
we set $\varepsilon(n) = 1$ if $d(X_{0}, \omega_{n}X''_{n})>d(X_{0},\omega_{i}X''_{i})$ for all $0\leq i<n$, and $\varepsilon(n) = 0$ for otherwise.
For negative $n$, $\varepsilon$ is defined similarly.
We set $\varepsilon(0) := 1$.
Then for $n$ positive, we define $X_{n}$ inductively
\begin{subnumcases}
{X_{n} := }\nonumber
X''_{n} & if $\varepsilon(n) = 1$\\\nonumber
\omega_{n}^{-1}\omega_{n-1}X_{n-1}, & if $\varepsilon(n) = 0$.
\end{subnumcases}
We define $X_{n}$ for negative $n$ similarly.
These $X_{n}$ are in the $\delta$-thick part and $\omega_{n}X_{n}$ are located according to the order of $n$ on $\Gamma(\omega)$. 
Even with this modification, the distance between $\omega_{n}X_{0}$ and $\omega_{n}X_{n}$ grows sublinearly.
\begin{figure}[htp]
\begin{center}
\includegraphics[scale = 0.5]{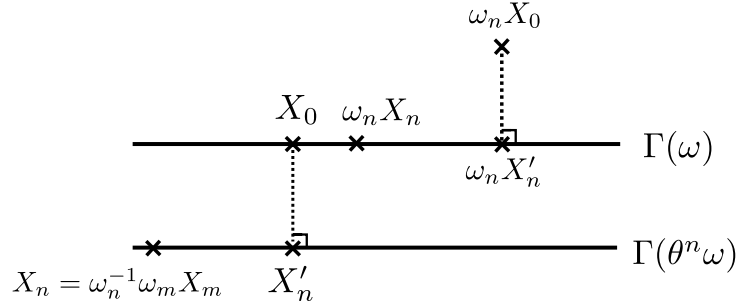}
\caption{Positions of $X_n$. The symbol of right angle is used to mean a closest point projection. 
Where $m:=\max\{k\in \mathbb{Z}_{+}\mid k<n\text{ and }\varepsilon(m) = 1\}$.}
\label{fig.position}
\end{center}
\end{figure}
\begin{lem}
\label{lem.Tio}
For above $\omega$ and $\{X_{n}\}$,  
$$\lim_{n\rightarrow\pm\infty}\frac{1}{n} d_{\mathcal{T}}(\omega_{n}X_{0}, \omega_{n}X_{n}) = 0.$$
\end{lem}
\begin{proof}
Since one can prove the statement for negative $n$ similarly, 
we assume that $n$ is positive. 
Then since $\nu$ and $\check\nu$ are independent, 
\begin{align*}
\mathbb{P}(\eta = (\eta_{n})\in\mathrm{MCG}(S)^{\mathbb{Z}}\mid F_{+}(\eta)\in U_{+} \text{ and } F_{-}(\eta)\in U_{-})
 = \nu(U_{+})\check\nu(U_{-})>0.
\end{align*}
Hence by the ergodic theorem $$\{n\in\mathbb{Z}: F_{+}(\theta^{n}\omega)\in U_{+} \text{ and }F_{-}(\theta^{n}\omega)\in U_{-}\}$$ has positive density.
We now recall the work of Tiozzo.
\begin{thm}[ {\cite[Theorem 18]{Tio}}]\label{thm.Tio}
For $\mathbb{P}$-a.e. $\omega\in\mathrm{MCG}(S)^{\mathbb{Z}}$, let $\Gamma$ denote the Teichm\"uller geodesic
 $\Gamma(\omega)$ with parametrization by arc length and $\Gamma(0) = X_{0}$.
 Then we have
$$\lim_{n\rightarrow\pm\infty}\frac{1}{n} d_{\mathcal{T}}(\omega_{n}X_{0}, \Gamma(Ln)) = 0$$
 where $L$ is the drift with respect to the Teichm\"uller metric.
\end{thm}

Let $$l_{n}(\omega):=\max\{k\mid \text{ for all } 0\leq n-k\leq i\leq n, F_{+}(\theta^{i}\omega)\not\in U_{+} \text{ or }F_{-}(\theta^{i}\omega)\not\in U_{-} \}.$$
Since $l_{n}(\omega)\leq n$,
$l_{n}$ is an integrable function.
Hence Kingman's subadditive ergodic theorem implies for $\mathbb{P}$-a.e. $\omega$,
$$\lim_{n\rightarrow\infty}l_{n}(\omega)/n = 0.$$
Then, we first estimate $d_{\mathcal{T}}(\Gamma(Ln),\omega_{n}X''_{n})$.
Let $\epsilon>0$.
By above observations, we may suppose that for large enough $n$,
we have $d_{\mathcal{T}}(\Gamma(Lm), \omega_{m}X'_{m})\leq m\epsilon\leq n\epsilon$ where $m:= n-l_{n}(\omega)$,  and 
$l_{n}(\omega)\leq n\epsilon$.
Then by definition, we have
\begin{align*}
d_{\mathcal{T}}(\Gamma(Ln),\omega_{n}X''_{n})
&\leq d_{\mathcal{T}}(\Gamma(Ln), \Gamma(Lm)) + d_{\mathcal{T}}(\Gamma(Lm), \omega_{m}X'_{m})\\
&\leq (L+1)n\epsilon 
\end{align*}
Hence we have $$\lim_{n\rightarrow\infty}d_{\mathcal{T}}(\Gamma(Ln),\omega_{n}X''_{n})/n = 0.$$
Now we estimate $d_{\mathcal{T}}(\omega_{n}X''_{n},\omega_{n}X_{n})$.
If $X''_{n}\not= X_{n}$, then there exists $m$ such that 
$d_{\mathcal{T}}(X_{0},\omega_{m}X''_{m})>d_{\mathcal{T}}(X_{0},\omega_{n}X''_{n})$.
Suppose $m$ is chosen to be maximum with this property so that $\omega_{m}X''_{m} = \omega_{n}X_{n}$.
Let $\epsilon>0$.
We may suppose that $n$ is large enough so that 
$d_{\mathcal{T}}(\Gamma(Lk),\omega_{k}X''_{k})\leq k\epsilon$ for $k\in\{n,m\}$.
Then we have two cases; $d_{\mathcal{T}}(X_{0},\Gamma(Ln))> d_{\mathcal{T}}(X_{0},\omega_{m}X''_{m})$
 or $d_{\mathcal{T}}(X_{0},\Gamma(Ln))\leq d_{\mathcal{T}}(X_{0},\omega_{m}X''_{m})$.
Since $Lm<Ln$, in both cases we have,
$$d_{\mathcal{T}}(\omega_{n}X''_{n}, \omega_{n}X_{n}) = d_{\mathcal{T}}(\omega_{n}X''_{n}, \omega_{m}X''_{m})
\leq (n+m)\epsilon\leq 2n\epsilon.$$
Hence we have $$\lim_{n\rightarrow\infty}d_{\mathcal{T}}(\omega_{n}X''_{n}, \omega_{n}X_{n})/n=0.$$
By Theorem \ref{thm.Tio} and the triangle inequality, we have the conclusion.

\end{proof}

Let $f:S\rightarrow \mathcal{X}_{0}$ be a representative of $X_{0}$.
Since each $\omega_{n}X_{n}$ is on $\Gamma(\omega)$,
there is the Teichm\"uller map $T_{n}$ that stretches $F_{+}(\omega)$ and contracts $F_{-}(\omega)$ such that
$f:S\rightarrow \mathcal{X}_{n}:= T_{n}(\mathcal{X}_{0})$ represents $\omega_{n}X_{n}$.
Then let $\mathcal{R}_{n}':=\mathcal{R}(F_{+}(\omega), F_{-}(\omega), \mathcal{X}_{n})$,
 $\tau_{n}:=\tau(F_{+}(\omega),F_{-}(\omega),\mathcal{X}_{n})$
and $\eta_{n}:=\eta(F_{+}(\omega),F_{-}(\omega),\mathcal{X}_{n})$.
We denote the corresponding measured foliation representatives of $F_{+}(\omega)$ and $F_{-}(\omega)$
 by $(\mathcal{F}_{+}(\omega,n),\mu_{+}(\omega,n))$ and $(\mathcal{F}_{-}(\omega,n),\mu_{-}(\omega,n))$
 respectively.
Since $\omega_{n}X_{n}$ are on $\Gamma(\omega)$, by the definition of $\varepsilon$, $\{w_{n}^{-1}\mathcal{R}'_{n}\}$ satisfies (M1).
We need to decompose each birectangles in $w_{n}^{-1}\mathcal{R}'_{n}$ further to have a partition which satisfies (M2).

Given two birectangle partitions $\mathcal{R},\mathcal{R}'$ with 
$\partial_{h}\mathcal{R}\subset \partial_{h}\mathcal{R}'$ and $\partial_{v}\mathcal{R}\supset\partial_{v}\mathcal{R}'$,
let $\mathcal{R}\vee\mathcal{R}'$ denote the birectangle partition 
we get by cutting $S$ by $\partial_{h}\mathcal{R}'\cup \partial_{v}\mathcal{R}$.
Let $0<i<j<k$ be indices which satisfy
\begin{enumerate}
\item $\varepsilon(i) = \varepsilon(j) = \varepsilon(k) = 1$, and
\item $\varepsilon(l)$ = 0 for all $i<l<j$ or $j< l<k$, 
\end{enumerate}
We define $\mathcal{R}_{j} := w_{j}^{-1}(\mathcal{R}'_{i} \vee \mathcal{R}'_{j} \vee \mathcal{R}'_{k})$.
We note that $\mathcal{R}'_{i} \vee \mathcal{R}'_{j} \vee \mathcal{R}'_{k}$
is equal to  
$$\{\bar{C}\mid C \text{ is a component of } S\setminus(w_{k}\tau_{k}\cup w_{i}\eta_{i})\}.$$
For $n > 0$ with $\varepsilon(n) = 0$, let $m$ be the largest integer which is less than $n$ and $\varepsilon(m) = 1$.
We define $\mathcal{R}_{n}:=w_{n}^{-1}w_{m}\mathcal{R}_{m}$.
For negative $n$, $\mathcal{R}_{n}$ is defined similarly.
By the construction, $\{\mathcal{R}_{n}\}$ still satisfies (M1).
\begin{lem}\label{lem.Markov}
$\{\mathcal{R}_{n}\}$ is a semi-Markov partition with respect to $\bold{w} = (w_{n})$.
\end{lem}
\begin{proof}
If $\varepsilon(n) = 0$, the condition (M2) is apparently satisfied.
Hence it suffices to prove for $m$ and $n$ with 
\begin{itemize}
\item $\varepsilon(m) = \varepsilon(n) = 1$ and
\item  $\varepsilon(l) = 0$ for all $m<l<n$, 
\end{itemize}
that for each $R_{m}\in\mathcal{R}_{m}$ and $R_{n}\in\mathcal{R}_{n}$, the intersection $w_{m}R_{m}\cap w_{n}R_{n}$ is either empty or a single birectangle.
Since $\{\mathcal{R}_{n}\}$ satisfies (M1), we see that if the intersection $w_{m}R_{m}\cap w_{n}R_{n}\not=\emptyset$, it is a family of birectangles.
Note that each birectangle in 
$w_{m}\mathcal{R}_{m}$ or $w_{n}\mathcal{R}_{n}$  
 is a subset of a component of $R'_{m}\cap R'_{n}$ for some $R'_{m}\in\mathcal{R}'_{m}$ and $R'_{n}\in\mathcal{R}'_{n}$.
 From each component $R'$ of the intersection $R'_{{m}}\cap R'_{n}$, 
a birectangle $w_{m}R_{m}\in w_{m}\mathcal{R}_{m}$ (resp. $w_{n}R_{n}\in \mathcal{R}_{n}$) is obtained by decomposing $R'$ vertically by leaves of $F_{-}$
 (resp. horizontally by leaves of $F_{+}$).
Hence each $w_{m}R_{m}\cap w_{n}R_{n}$ is connected.
Thus (M2) follows.
\end{proof}

\subsection{Symbolic dynamics}
We now associate a random subshift of finite type to the representative $\bold{w}$ of $\omega$
 by using the semi-Markov partition $\{\mathcal{R}_{n}\}$ constructed in \S \ref{subsec.Markov}.
Let $k(n,\bold{w})$ denote the number of birectangles in $\mathcal{R}_{n}$.
We label birectangles in $\mathcal{R}_{n}$ by $R^n_{1}, R^n_{2}, \dots, R^n_{k(n,\bold{w})}$.
We define $k(n,\bold{w})\times k(n+1,\bold{w})$ matrices $A(\bold{w},n) = (a^{n}_{i,j})$ by setting 
$a^{n}_{i,j} = 1$ if $w_{n}(\mathrm{int}(R^{n}_{i}))\cap w_{n+1}(\mathrm{int}(R^{n+1}_{j})) \not= \emptyset$ and
$a^{n}_{i,j} = 0$ for otherwise.
Let $(\{\Sigma_{A}(\bold{w},n)\}_{n\in\mathbb{Z}},\sigma)$ be the random subshift of finite type with respect to $\{A(\bold{w},n)\}_{n\in\mathbb{Z}}$.
Then each element in $\Sigma_{A}(\bold{w},n)$ corresponds to a point in $S$.
\begin{lem}\label{lem.defi-p}
For any $n$ and $\bold{b} = (b_{i})\in\Sigma_{A}(\bold{w},n)$,
$$\bigcap_{i=-\infty}^{\infty} w_{i+n}(\mathrm{int}(R^{i+n}_{b_{i}}))$$
 determines a single point in $S$.
\end{lem}
\begin{proof}
Let us fix $\bold{b} = (b_{i})\in\Sigma_{A}(\bold{w},n)$.
By the properties (M1) and (M2) of semi-Markov partitions, we have that for each $m$, $C_{m}:=\bigcap_{i=-m}^{m} w_{i+n}(\mathrm{int}(R^{i+n}_{b_{i}}))$ is a birectangle with exactly one component.
We consider the singular Euclidean metric that determines the point $\omega_{n}X_{n}$ on $\Gamma(\omega)$.
Let $\Gamma$ denote $\Gamma(\omega)$ with parametrization by arc length so that $\Gamma(0)=\omega_{n}X_{n}$.
We will prove that the diameter of $C_{m}$ converges to $0$ as $m\rightarrow-\infty$.
By Lemma \ref{lem.Tio}, we see that points $\omega_{m}X_{m}$ for negative $m$ are close to $\Gamma(L(n-m))$.
To construct $\{\mathcal{R}_{m}\}$, we considered arcs on 
$\mathcal{F}_{+}(\omega,m)$ of $\mu_{+}(\omega,m)$ measure 1
which is $\mu_{+}(\omega,n)$ measure almost equal to $1/\exp(L(n-m))$ by Lemma \ref{lem.Tio} and Theorem \ref{thm.Tio}.
Hence the horizontal diameter of $C_{m}$ converges to $0$ as $m\rightarrow -\infty$.
On the other hand, for $m$ positive the arcs $\tau(F_{+}(\omega),F_{-}(\omega),\mathcal{X}_{m})$ travel on singular leaves of $F_{+}$ longer as $m$ increases.
Since each infinite singular leaf is dense, it follows that the vertical diameter converges to $0$ as $m\rightarrow\infty$.
Thus we have a point on $\mathcal{X}_{n}$.
Finally by the marking $f\circ w_{n}:S\rightarrow \mathcal{X}_{n}$, we fixed above, we have a point on $S$.
\end{proof}
By Lemma \ref{lem.defi-p}, we define $p(\bold{w},n):\Sigma_{A}(\bold{w},n)\rightarrow S$.
\begin{lem}[c.f.{\cite[\S 10.4]{FLP}}]\label{lem.p-surj}
The map $p(\bold{w},n)$ is continuous and surjective.
\end{lem}
\begin{proof}
Since the image of a long cylinder set of $\Sigma_{A}(\bold{w},n)$ is contained in a small birectangle,
$p(\bold{w},n)$ is continuous.
Let $V_{j} = \bigcup_{i=1}^{k(j,\bold{w})} w_{j}(\mathrm{int}(R_{i}^{j}))$.
For each $j$, $V_{j}$ is an open dense set.
Then by the Baire category theorem, 
$U := \bigcap_{j\in\mathbb{Z}}V_{j}$
is dense.
Each $x\in U$ is contained in $w_{j+n}(\mathrm{int}(R^{j+n}_{b_{j}}))$ for some $b_{j}$ for every $j\in\mathbb{Z}$.
Let $\bold{b} = \{b_{j}\}_{j\in\mathbb{Z}}\in\Sigma_{A}(\bold{w},n)$.
We have $p(\bold{w},n)(\bold{b}) = x$, which implies $U\subset p(\bold{w},n)(\Sigma_{A}(\bold{w},n))$.
Since $U$ is dense and $\Sigma_{A}(\bold{w},n)$ is compact, $p(\bold{w},n)$ is surjective.
\end{proof}

We are now ready to prove Theorem \ref{thm.symbolic}.
\begin{proof}[{Proof of Theorem \ref{thm.symbolic}}]
Note that $T_{n}:\mathcal{X}_{0}\rightarrow\mathcal{X}_{n}$ changes only the metric and does not change the image.
Since $p(\bold{w},n):\Sigma_{A}(\bold{w},n)\rightarrow S$ is defined by using $f\circ w_{n}:S\rightarrow \mathcal{X}_{n}$,
we have $p(\bold{w},n) \circ \sigma^{n} = w_{n}^{-1}\circ p(\bold{w},0)$.
\end{proof}

We now consider the topological entropy of the shift map $\sigma:\Sigma_{A}(\bold{w},n)\rightarrow\Sigma_{A}(\bold{w},n+1)$.
In order to prove Theorem \ref{thm.main}, it suffices to prove that growth ratio of the number of elements of cylinder sets.
\begin{lem}\label{lem.p-entropy}
Let $(\{\Sigma_{A}(\bold{w},n)\}_{n\in\mathbb{Z}},\sigma)$ be the random subshift of finite type defined above.
Then for any $K\in\mathbb{Z}$,
$$ \limsup_{m\rightarrow\infty}\frac{\log N(\mathcal{C}_{n}(K,K+m))}{m} \leq L,$$
where $L$ is the drift of $\omega$ with respect to the Teichm\"uller metric.
\end{lem}
\begin{proof}
Note that by the property (M1) and (M2) of semi-Markov partitions, the intersections 
$\mathcal{R}_{K}\vee\cdots\vee\mathcal{R}_{K+m}$ is also a birectangle partition.
For a given birectangle partition $\mathcal{R}$, let $N(\mathcal{R})$ denotes the number of birectangles.
By the map $p(\bold{w},n)$, we see that 
$$N(\mathcal{C}_{n}(K,K+m)) = N(\mathcal{R}_{K}\vee\cdots\vee\mathcal{R}_{K+m}).$$
Hence we will give a bound of $N(\mathcal{R}_{K}\vee\cdots\vee\mathcal{R}_{K+m})$.
Let $c_{K}$ be the shortest horizontal length of birectangles in $\mathcal{R}_{K}$ measured by $\mu_{-}(\omega,n)$.
Let $L_{m}$ denote the maximum of $\mu_{-}(\omega,n)$ measures of the arcs $\tau_{K+m}:=\tau(F_{+}(\omega),F_{-}(\omega),\mathcal{X}_{K+m})$.
Each arc in  $\tau_{K+m}$ cuts birectangles in $\mathcal{R}_{K}$ at most ${L_{m}}/c_{K}$ times.
The number of singular leaves of $F_{+}(\bold{w},n)$ is bounded from above by some
constant $D$ which depends only on $S$.
Hence  $N(\mathcal{R}_{K}\vee\cdots\vee\mathcal{R}_{K+m})$ is at most
$D\cdot {L_{m}}/c_{K}$.
By Lemma \ref{lem.Tio} and Theorem \ref{thm.Tio},
$$\lim_{m\rightarrow\infty} \frac{1}{m}\log L_{m} = L,$$
which implies
$$\limsup_{m\rightarrow\infty}\log (N(\mathcal{R}_{K}\vee\cdots\vee\mathcal{R}_{K+m})) \leq L.$$
\end{proof}

\section{Proof of the main theorem}\label{sec.main}
In \S \ref{sec.Markov}, we have constructed a semi-Markov partition $\{\mathcal{R}_{n}\}$ for any representative of $\mathbb{P}$-a.e. $\omega\in\mathrm{MCG}(S)^{\mathbb{Z}}$.
In this section, we will prove that for $\mathbb{P}$-a.e. $\omega\in\mathrm{MCG}(S)^{\mathbb{Z}}$, we can find a representative $\bold{w} = (w_{i})$ of $\omega$
such that $h(\bold w) \leq L$.
In the case of pseudo-Anosov diffeomorphisms, facts of type Lemma \ref{lem.p-surj} and \ref{lem.p-entropy} suffice to prove that the topological entropy and 
the translation distance on the Teichm\"uller space agree.
This is because to construct a Markov partition for a pseudo-Anosov diffeomorphism, we only need to use a single point in $\mathcal{T}(S)$.
On the other hand,
for random walks, we need to use different $X_{n}$'s for each $n$.
One of the main difficulty caused for above reason is that the Lebesgue number of a given open covering $\mathcal{A}$ varies depending on the metric.
In \S \ref{sec.Lebesgue}, we first give a suitable asymptotic bound for the Lebesgue number.
Every argument so far works for any $U_{\pm}$ satisfying \ref{condi.U}.
In \S \ref{sec.positive-nu}-\ref{sec.refine}, the neighborhoods $U_{\pm}$ of $\mathcal{F}_{\pm}(\omega)$ 
which we need to prove Theorem \ref{thm.main} are given.
\subsection{Bound for the Lebesgue number}\label{sec.Lebesgue}
To have a bound of the Lebesgue number, we first observe how singular Euclid structures may change in the $\delta$-thick part of $\mathcal{T}(S)$.
\begin{lem}\label{lem.bilipschiz}
There exists $B = B(S,\delta)$ such that the following holds.
Let $X$ be in the $\delta$-thick part of $\mathcal{T}(S)$, and  $q_{1}, q_{2}\in Q_{1}(X)$.
Then the singular Euclidean metric associated to $q_{1}$ and $q_{2}$ are $B$-bi-Lipschitz.
\end{lem}
\begin{proof}
As pointed out in \cite[Lemma 9.22]{FLP}, any two singular Euclidean metrics are bi-Lipschitz with some bi-Lipschitz constant.
Since by Theorem \ref{thm.HM}, $Q_{1}(X)$ is homeomorphic to $\mathcal{PMF}(S)$, a compact space,
we see that for any $X\in\mathcal{T}(S)$,
 there exists $B(X)>0$ such that any two singular Euclidean metrics corresponding to elements in $Q_{1}(X)$ are $B(X)$-bi-Lipschitz.
Since $B(X)$ varies continuously and the $\delta$-thick part of the moduli space of $S$ is compact, we have a desired bound.
\end{proof}

Recall that the semi-Markov partition $\{\mathcal{R}_{n}\}$ is defined on $X_{n}\in\mathcal{T}(S)$ with 
representative $f\circ w_{n}:S\rightarrow\mathcal{X}_{n}$.
Then there is a quadratic differential $q_{n}\in Q(\mathcal{X}_{0})$ that is the initial quadratic differential of the 
Teichm\"uller geodesic connecting $X_{0}$ and $X_{n}$.
Let $\mathcal{X}'_{n}$ denote the complex structure we get by stretching (resp. contracting) horizontal (resp. vertical) foliation of $q_{n}$ so that it gives the same point as $X_{n}$ in $\mathcal{T}(S)$.
Let $T_{n}'$ be the corresponding Teichm\"uller map.
Since two markings $f\circ w_{n}:S\rightarrow \mathcal{X}_{n}$ and $f:S\rightarrow\mathcal{X}_{n}'$
gives the same point in the Teichm\"uller space, there is a biholomorphic map 
$\phi:\mathcal{X}_{n}'\rightarrow\mathcal{X}_{n}$ so that $\phi\circ f\circ w_{n}$ is homotopic to $f$.
Hence by homotopy, we may suppose $w_{n} = f^{-1}\circ \phi^{-1}\circ f$.
From now on, we use these representations and let $\bold{w} = (w_{n})$.
We now fix an open covering $\mathcal{A}$ of $S$.
Let $q'_{n}$ be the quadratic differential on $S$ determined by $\mathcal{X}_{n}'$ and $\Gamma(\theta^{n}\omega)$.
For a quadratic differential $q$, 
we denote by $\delta(q)$ the Lebesgue number of $\mathcal{A}$ with respect to the singular Euclidean metric defined by $q$.
By the choice of the representative $\bold{w}=(w_{n})$, 
$\delta(q'_{n})$ is equal to the Lebesgue number of $w_{n}\mathcal{A}$ with respect to the quadratic differential determined by $\Gamma(\omega)$ and $\mathcal{X}_{n}$ 
that we used to construct $\mathcal{R}(F_{+}(\omega),F_{-}(\omega),\mathcal{X}_{n})$.
\begin{lem}\label{lem.Lebesgue}
For $\mathbb{P}$-a.e. $\omega$, the $\{q'_{n}\}$ defined above satisfies
$$\lim_{n\rightarrow\infty} \frac{-\log\delta(q'_{n})}{n} = 0.$$
\end{lem}
\begin{proof}
We first note that if two singular Euclidean metrics determined by $q$ and $q'$ are $B$-bi-Lipschitz, then we have $\delta(q)/\delta(q')\leq B$.
Since we have chosen $X_{n}$ so that they are in the $\delta$-thick part, we have $\delta(q'_{0})/\delta(q_{n})\leq B$ 
and $\delta(T'_{n}(q_{n}))/\delta(q'_{n})\leq B$ by Lemma \ref{lem.bilipschiz}.
By the definition of $T'_{n}$,
the ratio $\delta(q_{n})/\delta(T'_{n}(q_{n}))$ is bounded from above by $\exp(d_{\mathcal{T}}(X_{0},X_{n}))$.
Hence $$1/\delta(q'_{n})\leq B^{2}\exp(d_{\mathcal{T}}(X_{0},X_{n}))/\delta(q'_{0}).$$
Therefore, by Lemma \ref{lem.Tio}, we have the conclusion.
\end{proof}

\subsection{$\nu$-measures of neighborhoods of $F_{+}(\omega)$.}\label{sec.positive-nu}
The goal of this subsection is the following proposition.
The measure $\nu$ is from Theorem \ref{thm.KM}.
\begin{prop}\label{prop.positive-nu}
For $\mathbb{P}$-a.e. $\omega$, any open neighborhood $U$ of $F_{+}(\omega)$ has positive $\nu$-measure.
\end{prop}
We first recall the curve graphs and shadows.
The {\em curve graph} of $S$, denoted $\mathcal{C}(S)$, is the graph whose set of vertices are the set of isotopy classes of essential simple closed curves,
and two vertices are connected by an edge of length $1$ if corresponding simple closed curves can be represented disjointly.
For $x,y,z\in\mathcal{C}(S)$, the {\em Gromov product} of $y$ and $z$ with respect to $x$, denoted $(y\cdot z)_{x}$ is defined by

 $$(y\cdot z)_{x} := \frac{1}{2}(d(x,y)+d(x,z) - d(y,z)).$$
Since $\mathcal{C}(S)$ is Gromov hyperbolic by \cite{MM},
it has the Gromov boundary $\partial \mathcal{C}(S)$.
Let $\bar{\mathcal{C}}(S) := \mathcal{C}(S)\cup \partial\mathcal{C}(S)$.
A sequence of points $\{x_{i}\in\bar{\mathcal{C}}(S)\}$ converges to a point $\lambda\in\partial\mathcal{C}(S)$ if 
$(x_{i}\cdot\lambda)_{x}\rightarrow\infty$.
We define a {\em shadow set} by $$S_{x}(y,R):= \{z\in\bar{\mathcal{C}}(S)\mid (y\cdot z)_{x}\geq R\}.$$
\begin{proof}[Proof of Proposition \ref{prop.positive-nu}]
By \cite[Theorem 1.2 and 1.4]{Kla} and the fact that $F_{+}(\omega)$ is uniquely ergodic,
 we see that $S_{x}(F_{+}(\omega),R)\subset U$ as subsets of $\mathcal{PMF}(S)$
 for sufficiently large $R$.

Since $\omega_{n}x$ converges to $F_{+}(\omega)$  in $\bar{\mathcal{C}}(S)$, 
we see that for any $D$, there exists $N\in\mathbb{Z}_{+}$ such that for any $n>N$,
 $\omega_{n}x\in S_{x}(F_{+}(\omega), R+D)$.
Hence by the work of Maher \cite[Proposition 2.13 (5)]{Mah2012},
for any $\epsilon>0$,  
there is $\omega_{n}$ such that
$\nu_{\omega_{n}}(S_x(F_{+}(\omega),R))\geq 1-\epsilon$ where
$\nu_{\omega_{n}}(A):= \nu(\omega_{n}^{-1}A)$.
Note that in \cite{Mah2012}, the measure $\mu$ is assumed to have a finite support, however, the finiteness is not used for the proof of results we need here.
Hence 
$$\nu(S_{x}(F_{+}(\omega),R))\geq \mu_{n}(\omega_{n})\nu_{\omega_{n}}(S_x(F_{+}(\omega),R))\geq
\mu_{n}(\omega_{n})(1-\epsilon)>0,$$
where $\mu_{n}$ is the $n$-fold convolution of $\mu$.
Since $U\supset S_{x}(F_{+}(\omega),R)$, we have $\nu(U)>0$.
\end{proof}

\subsection{Refinement by cylinders}\label{sec.refine}
We finally give neighborhoods $U_{+}$ and $U_{-}$ of $F_{+}(\omega)$ and $F_{-}(\omega)$ respectively.
We would like to find $U_{\pm}$ so that the vertical lengths of birectangles in $\{\mathcal{R}_{n}\}$ are bounded from above.
Recall that given two $G_{{\pm}}\in\mathcal{UE}(S)\subset\mathcal{PMF}(S)$, we can construct a birectangle decomposition $\mathcal{R}({G_{+},G_{-},\mathcal{X}_{G}})$, where $f:S\rightarrow\mathcal{X}_{G}$ is a representative of a closest point projection $X_{G}$ on $\Gamma(G_{+},G_{-})$ of $X_{0}$.
Since the vertical length is independent of representatives of points of the Teichm\"uller space, we denote birectangle partitions by $\mathcal{R}(G_{+},G_{-},X_{G})$.
We abuse notations similarly for $\tau$ and $\eta$.
Let $V(\mathcal{R}(G_{+},G_{-},X_{G}))$ denote the maximum of the vertical lengths of birectangles in $\mathcal{R}({G_{+},G_{-},X_{G}})$.
Since each $\mathcal{R}_{n}$ is obtained from some $\mathcal{R}(G_{+},G_{-},X_{G})$ by decomposing each birectangle in $\mathcal{R}(G_{+},G_{-},X_{G})$, it suffices to prove the following.
\begin{lem}\label{lem.U}
There exist $V>0$ and open neighborhoods $U_{+}$ and $U_{-}$ of $F_{+}(\omega)$ and $F_{-}(\omega)$ respectively so that the following holds.
Suppose that 
\begin{itemize}
\item[($\ast$)]$G_{\pm}\in U_{\pm}$ are written as $G_{\pm} = \eta F_{\pm}(\omega)$ 
respectively for some $\eta\in\mathrm{MCG}(S)$. 
\end{itemize}
Then $V(\mathcal{R}(G_{+}, G_{-}, X_{G}))<V$.
\end{lem}
\begin{proof}
Suppose the contrary.
Then there are $V_{n}$ with $V_{n}\rightarrow\infty$ and 
$\eta_{n}\in\mathrm{MCG}(S)$ with $G_{\pm}^{n}:=\eta_{n}F_{\pm}(\omega)\rightarrow F_{\pm}(\omega)$ such that
$V(\mathcal{R}(G^{n}_{+}, G^{n}_{-},X_{G^{n}})) = V_{n}$.
Let $R_{n}\in\mathcal{R}(G^{n}_{+}, G^{n}_{-},X_{G^{n}})$ be the rectangle with vertical length $V_{n}$.
Note that since we assume that the total area is equal to $1$, the horizontal length of $R_{n}$ converges to $0$.
Hence by taking a subsequence if necessary, 
the vertical boundary $\partial_{v}R_{n}$ converges in the Hausdorff topology to 
an infinite subarc of a singular leaf of $F_{-}(\omega)$, which intersects $\tau_{0}$ only twice.
However any infinite singular leaf of $F_{-}(\omega)$ is dense, so such a singular leaf never exists.
\end{proof}

Note that since $U_{+}$ and $U_{-}$ are open,
we see that $\nu(U_{+})>0$ and $\check\nu(U_{-})>0$ by Proposition \ref{prop.positive-nu}.
By taking smaller open neighborhoods if necessary, we may suppose that $U_{\pm}$ in Lemma \ref{lem.U} satisfy Condition \ref{condi.U}. 
From now on, we consider the semi-Markov partition $\{\mathcal{R}_{n}\}$ constructed with such $U_{\pm}$
and corresponding representation $\bold{w}$ of $\omega$ whose construction is given in \S \ref{sec.Lebesgue}.
We now consider images of cylinders in $\Sigma_{A}(\bold{w},n)$ by $p(\bold{w},n)$.
For notational simplicity, we call the image of cylinders by $p(\bold{w},n)$ cylinders and omit to write $p(\bold{w},n)$.
Let $c(n)$ be the number so that the set of cylinders $\mathcal{C}_{n}(-c(n),c(n))$ refines $\mathcal{A}$.
The existence of $c(n)$ follows from Lemma \ref{lem.defi-p}.
We now prove that $c(n)$ grows sublinearly.
\begin{lem}\label{lem.c(n)}
For the above $c(n)$, we have
$$\lim_{n\rightarrow\infty} \frac{c(n)}{n} = 0.$$
\end{lem}
\begin{proof}

It suffices to find $c(n)$ so that the horizontal and the vertical lengths with respect to $q'_{n}$
of $(-c(n),c(n))_{n}$-cylinders are less than $\delta(q'_{n})/\sqrt{2}$.
Let $V>0$ be the upper bound given by Lemma \ref{lem.U}.
Then the horizontal and vertical length of any $(-c(n),c(n))_{n}$-cylinder is bounded from above by
$1/\exp(d_{\mathcal{T}}(w_{n}X_{n},w_{n-c(n)}X_{n-c(n)}))$ 
and $V/\exp(d_{\mathcal{T}}(w_{n}X_{n},w_{n+c(n)}X_{n+c(n)}))$ respectively.
Since the bound for the horizontal length is given similarly, we only discuss the vertical lengths.
Let $m:=n+c(n)$ and let $\epsilon>0$ be arbitrary.
By Lemma \ref{lem.Tio}  and Theorem \ref{thm.Tio}, for sufficiently large $n$, we have 
$d_{\mathcal{T}}(w_{n}X_{n},w_{m}X_{m})\geq (m-n)L-(n+m)\epsilon$.
We also have $\delta(q'_{n})\geq1/ \exp (n\epsilon)$ by Lemma \ref{lem.Lebesgue}.
Let $c^{\epsilon}(n)$ be the smallest integer with $c^{\epsilon}(n) > (\log(\sqrt{2}V) + 3n\epsilon)/(L-\epsilon)$.
Then for large enough $n$, $\mathcal{C}_{n}(-c^{\epsilon}(n), c^{\epsilon}(n))$ refines $\mathcal{A}$.
Hence for large enough $n$, $c(n)\leq c^{\epsilon}(n)$.
Since $\epsilon>0$ is arbitrary,  we have $\lim_{n\rightarrow\infty}c(n)/n=0$.
\end{proof}

We are now ready to prove the main theorem.
\begin{proof}[{Proof of Theorem \ref{thm.main}}]
Since $\lim_{n\rightarrow\infty}c(n)/n = 0$, there exists $K>0$ such that
$-K<-c(n)+n$ for all $n\in\mathbb{N}$.
Therefore we have,
\begin{align*}
&\frac{1}{n+1}\log N(\mathcal{A}\vee w_{1}\mathcal{A}\vee \cdots\vee w_{n}\mathcal{A})\\
&\leq \frac{1}{n+1}\log N(\mathcal{C}_{0}(-c(0),c(0))\vee \cdots \vee w_{n}\mathcal{C}_{n}(-c(n),c(n))&\text{(by definition of $c(n)$)}\\
&= \frac{1}{n+1}\log N(\mathcal{C}_{0}(-c(0), c(0)) \vee \cdots\vee\sigma^{-n}\mathcal{C}_{0}(-c(n)+n, n+c(n)))&\text{(Theorem \ref{thm.symbolic})}\\
&\leq \frac{1}{n+1}\log N(\mathcal{C}_{0}(-K,n+c(n)))\\
&= \frac{\log N(\mathcal{C}(-K,n+c(n)))}{n+c(n)+K}\cdot(\frac{n+c(n)+K}{n+1}).
\end{align*} 
By Lemma \ref{lem.p-entropy}, we have
$$ \limsup_{n\rightarrow\infty}\frac{\log N(\mathcal{C}(-K,n+c(n)))}{n+c(n)+K} \leq L.$$
Also by Lemma \ref{lem.c(n)},
$$\lim_{n\rightarrow\infty}(\frac{n+c(n)+K}{n+1}) = 1.$$
Therefore,
$$\limsup_{n\rightarrow\infty}\frac{\log N(\mathcal{C}(-K,n+c(n)))}{n+c(n)+K}\cdot(\frac{n+c(n)+K}{n+1}) \leq L.$$
By putting all estimates together, we have
\begin{align*}
&\limsup_{n\rightarrow\infty}\frac{1}{n}\log N(\mathcal{A}\vee w_{1}\mathcal{A}\vee \cdots\vee w_{n-1}\mathcal{A})
\leq L.
\end{align*}
Hence the representative $\bold{w}$ satisfy that for arbitrary open covering $\mathcal{A}$,  
$h(\bold{w},\mathcal{A})\leq L$.
Putting together with Theorem \ref{thm.Kar} and Lemma \ref{lem.lemma1}, we have $h(\omega) = L$.
\end{proof}

\section*{Acknowledgement}
This work was partially supported by JSPS Research Fellowship for Young Scientists.

%


\end{document}